\documentclass[12pt]{amsart}
\usepackage{amscd,amsmath,amsthm,amssymb}
\usepackage[left]{lineno}
\usepackage{pstcol,pst-plot,pst-3d}
\usepackage{color}
\usepackage{pstricks}
\usepackage{stmaryrd}
\newpsstyle{fatline}{linewidth=1.5pt}
\newpsstyle{fyp}{fillstyle=solid,fillcolor=verylight}
\definecolor{verylight}{gray}{0.97}
\definecolor{light}{gray}{0.9}
\definecolor{medium}{gray}{0.85}
\definecolor{dark}{gray}{0.6}

 \def\NZQ{\mathbb}               

 \def\ZZ{{\NZQ Z}}
 \def\RR{{\NZQ R}}

 \def\G{{\mathcal G}}

 \def\P{{\mathcal P}}

 \def\ab{{\mathbf a}}

 \def\eb{{\mathbf e}}
 \def\opn#1#2{\def#1{\operatorname{#2}}} 
 \opn\chara{char} \opn\length{\ell} \opn\pd{pd} \opn\rk{rk}
 \opn\projdim{proj\,dim} \opn\injdim{inj\,dim} \opn\rank{rank}
 \opn\depth{depth} \opn\grade{grade} \opn\height{height}
 \opn\embdim{emb\,dim} \opn\codim{codim}
 
 \opn\Tr{Tr} \opn\bigrank{big\,rank}
 \opn\superheight{superheight}\opn\lcm{lcm}
 \opn\trdeg{tr\,deg}
 \opn\reg{reg} \opn\lreg{lreg} \opn\ini{in} \opn\lpd{lpd}
 \opn\size{size} \opn\sdepth{sdepth}
 \opn\link{link}\opn\fdepth{fdepth}\opn\lex{lex}
 \opn\tr{tr}
 \opn\type{type}
 \opn\gap{gap}
 \opn\arithdeg{arith-deg}
 \opn\astab{astab}
  \opn\dstab{dstab}
  \opn\pol{pol}
 \opn\div{div} \opn\Div{Div} \opn\cl{cl} \opn\Cl{Cl}
 \opn\Spec{Spec} \opn\Supp{Supp} \opn\supp{supp} \opn\Sing{Sing}
 \opn\Ass{Ass} \opn\Min{Min}\opn\Mon{Mon}
 \opn\Ann{Ann} \opn\Rad{Rad} \opn\Soc{Soc}
 \opn\Im{Im} \opn\Ker{Ker} \opn\Coker{Coker} \opn\Am{Am}
 \opn\Hom{Hom} \opn\Tor{Tor} \opn\Ext{Ext} \opn\End{End}
 \opn\Aut{Aut} \opn\id{id}
 
 \opn\nat{nat}
 \opn\pff{pf}
 \opn\Pf{Pf} \opn\GL{GL} \opn\SL{SL} \opn\mod{mod} \opn\ord{ord}
 \opn\Gin{Gin} \opn\Hilb{Hilb}\opn\sort{sort}
 \opn\PF{PF}\opn\Ap{Ap}
 \opn\mult{mult}
 \opn\aff{aff}
 \opn\relint{relint} \opn\st{st}
 \opn\lk{lk} \opn\cn{cn} \opn\core{core} \opn\vol{vol}  \opn\inp{inp} \opn\nilpot{nilpot}
 \opn\link{link} \opn\star{star}\opn\lex{lex}\opn\set{set}
 \opn\width{wd}
 \opn\Fr{F}
 \opn\QF{QF}
 \opn\G{G}
 \opn\type{type}\opn\res{res}
 \opn\conv{conv}
 \opn\mat{mat}
 \opn\gr{gr}
 
 \def\pot#1#2{#1[\kern-0.28ex[#2]\kern-0.28ex]}

 \opn\dirlim{\underrightarrow{\lim}}
 \opn\inivlim{\underleftarrow{\lim}}
 \def\Implies{\ifmmode\Longrightarrow \else
         \unskip${}\Longrightarrow{}$\ignorespaces\fi}
 \def\implies{\ifmmode\Rightarrow \else
         \unskip${}\Rightarrow{}$\ignorespaces\fi}
 \def\iff{\ifmmode\Longleftrightarrow \else
         \unskip${}\Longleftrightarrow{}$\ignorespaces\fi}

 \let\:=\colon
 \newtheorem{Theorem}{Theorem}
 \newtheorem{Lemma}[Theorem]{Lemma}
 \newtheorem{Corollary}[Theorem]{Corollary}
 \newtheorem{Proposition}[Theorem]{Proposition}

 \newtheorem{Question}[Theorem]{Question}
 \let\epsilon\varepsilon
 \let\kappa=\varkappa
 \def\qed{\ifhmode\textqed\fi
       \ifmmode\ifinner\quad\qedsymbol\else\dispqed\fi\fi}
 \def\textqed{\unskip\nobreak\penalty50
        \hskip2em\hbox{}\nobreak\hfil\qedsymbol
        \parfillskip=0pt \finalhyphendemerits=0}
 \def\dispqed{\rlap{\qquad\qedsymbol}}
 
 \opn\dis{dis}
 \def\pnt{{\raise0.5mm\hbox{\large\bf.}}}
 
 \opn\Lex{Lex}

\begin{document}

\title {The regularity of edge rings and matching numbers}

\author{J\"urgen Herzog and Takayuki Hibi}

\address{J\"urgen Herzog, Fachbereich Mathematik, Universit\"at Duisburg-Essen, Campus Essen, 45117
Essen, Germany} \email{juergen.herzog@uni-essen.de}

\address{Takayuki Hibi, Department of Pure and Applied Mathematics, Graduate School of Information Science and Technology,
Osaka University, Toyonaka, Osaka 560-0043, Japan}
\email{hibi@math.sci.osaka-u.ac.jp}

\dedicatory{ }

\begin{abstract}
Let $K[G]$ denote the edge ring of a finite connected simple graph $G$ on $[d]$ and $\mat(G)$ the matching number of $G$.
It is shown that $\reg(K[G]) \leq \mat(G)$ if $G$ is non-bipartite and $K[G]$ is normal,  and that $\reg(K[G]) \leq \mat(G) - 1$ if $G$ is bipartite.
\end{abstract}

\thanks{The first author was partially supported by JSPS KAKENHI 19H00637.}



\maketitle

\setcounter{tocdepth}{1}

Let $G$ be a finite connected simple graph on the vertex set $[d] = \{1, \ldots, d\}$ and $E(G)$ its edge set.  Let $S = K[x_1, \ldots, x_d]$ denote the polynomial ring in $d$ variables over a field $K$.  The {\em edge ring} of $G$ is the toric ring $K[G] \subset S$ which is generated by those monomials $x_ix_j$ with $\{i, j \} \in E(G)$.  The systematic study of edge rings originated in \cite{OH}.  It is shown that $K[G]$ is normal if and only if $G$ satisfies the odd cycle condition \cite[p.~131]{GTM279}.  Thus in particular $K[G]$ is normal if $G$ is bipartite.

Let $\eb_1, \ldots, \eb_d$ denote the canonical unit coordinate vectors of $\RR^d$.  The {\em edge polytope} is the lattice polytope $\P_G \subset \RR^d$ which is the convex hull of the finite set $\{ \, \eb_i + \eb_j \, : \, \{i, j \} \in E(G) \, \}$.  One has $\dim \P_G = d - 1$ if $G$ is non-bipartite and $\dim \P_G = d - 2$ if $G$ is bipartite.  We refer the reader to \cite[Chapter 5]{GTM279} for the fundamental materials on edge rings and edge polytopes.

A {\em matching} of $G$ is a subset $M \subset E(G)$ for which $e \cap e' = \emptyset$ for $e \neq e'$ belonging to $M$.  The {\em matching number} is the maximal cardinality of matchings of $G$.  Let $\mat(G)$ denote the matching number of $G$.

When $K[G]$ is normal, the upper bound of regularity of $K[G]$ can be explicitly described in terms of $\mat(G)$.  In fact,

\begin{Theorem}
\label{oldhibi}
Let $G$ be a finite connected simple graph.  Then
\begin{enumerate}
\item[(a)] $\reg K[G]\leq \mat(G)$, if  $G$ is non-bipartite and  $K[G]$ is normal;

\item[(b)] $\reg K[G]\leq \mat(G)-1$, if $G$ is bipartite.
\end{enumerate}
\end{Theorem}

Lemma \ref{facet} stated below, which provides information on lattice points belonging to interiors of dilations of edge polytopes, is indispensable for the proof of Theorem \ref{oldhibi}.

\begin{Lemma}
\label{facet}
Suppose that $(a_1, \ldots, a_d) \in \ZZ^d$ belongs to the interior $q(\P_G \setminus \partial \P_G)$ of the dilation $q \P_G = \{ q \alpha : \alpha \in \P_G \}$, where $q \geq 1$, of $\P_G$.  Then each $a_i \geq 1$. 
\end{Lemma} 

\begin{proof}
The facets of $\P_G$ are described in \cite[Theorem 1.7]{OH}.  When $W \subset [d]$, we write $G_W$ for the induced subgraph of $G$ on $W$.  Since $K[G]$ is normal, it follows that $\P_G$ possesses the integer decomposition property \cite[p.~91]{GTM279}.  In other words, each $\ab \in q \P_G \cap \ZZ^d$ is of the form 
\[
\ab = (\eb_{i_1} + \eb_{j_1}) + \cdots + (\eb_{i_q} + \eb_{j_q}),
\]
where $\{i_1, j_1\}, \ldots, \{i_q, j_q\}$ are edges of $G$.

\medskip

\noindent
{\bf (First Step)}  Let $G$ be non-bipartite.  Let $i \in [d]$.  Let $H_1, \ldots, H_s$ and $H'_1, \ldots, H'_{s'}$ denote the connected components of $G_{[d] \setminus \{i\}}$, where each $H_j$ is bipartite and where each $H'_{j'}$ is non-bipartite.  If $s = 0$, then $i \in [d]$ is regular (\cite[p.~414]{OH}) and the hyperplane of $\RR^d$ defined by the equation $x_i = 0$ is a facet of $q\P_G$.  Hence $a_i > 0$.  

Let $s \geq 1$ and $s' \geq 0$.  Let $W_j \cup U_j$ denote the vertex set of the bipartite graph $H_j$ for which there is $a \in W_j$ with $\{a, i\} \in E(G)$, where $U_j = \emptyset$ if $H_j$ is a graph consisting of a single vertex.  Then $T = W_1 \cup \cdots \cup W_s$ is independent (\cite[p.~414]{OH}).  In other words, no edge $e \in E(G)$ satisfies $e \subset T$.  Let $G'$ denote the bipartite graph induced by $T$.  Thus the edges of $G'$ are $\{b, c\} \in E(G)$ with $b \in T$ and $c \in T' = U_1 \cup \cdots \cup U_s \cup \{i\}$.  Then $G'$ is connected with $V(G') = T \cup T'$ its vertex set.  Since the connected components of $G_{[d] \setminus V(G')}$ are $H'_1, \ldots, H'_{s'}$, it follows that $T$ is fundamental (\cite[p.~415]{OH}) and the hyperplane of $\RR^d$ defined by $\sum_{\xi \in T} x_{\xi} = \sum_{\xi' \in T'} x_{\xi'}$ is a facet of $q\P_G$.  Now, suppose that $a_i = 0$.  Since $\P_G$ possesses the integer decomposition property, one has $\sum_{\xi \in T} a_{\xi} = \sum_{\xi' \in T'} a_{\xi'}$.  Hence $(a_1, \ldots, a_d) \in \ZZ^d$ cannot belong to $q(\P_G \setminus \partial \P_G)$.  Thus $a_i > 0$, as desired.

\medskip

\noindent
{\bf (Second Step)}  Let $G$ be bipartite.  If $G$ is a star graph with, say, $E(G) = \{ \{1,2\}, \{1,3\}, \ldots, \{1,d\} \}$, then $\P_G$ can be regarded to be the $(d - 2)$ simplex of $\RR^{d-1}$ with the vertices $(1, 0, \ldots, 0), (0, 1, 0, \ldots, 0),\ldots, (0, \ldots, 0, 1)$.  Thus, since each $(a_1, \ldots, a_d) \in q\P_G \cap \ZZ^d$ satisfies $a_1 = q$, the assertion follows immediately.  In the argument below, one will assume that $G$ is not a star graph.  

Let $i \in [d]$ and $H_1, \ldots, H_s$ the connected components of $G_{[d] \setminus \{i\}}$.  If $s = 1$, then $i \in [d]$ is ordinary (\cite[p.~414]{OH}) and the hyperplane of $\RR^d$ defined by the equation $x_i = 0$ is a facet of $q\P_G$.  Hence $a_i > 0$.

Let $s \geq 2$.  Let $W_j \cup U_j$ denote the vertex set of $H_j$ for which there is $a \in W_j$ with $\{a, i\} \in E(G)$.  Since $G$ is not a star graph, one can assume that $U_1 \neq \emptyset$.  Then $T = W_2 \cup \cdots \cup W_s$ is independent and the bipartite graph induced by $T$ is $G_{[d] \setminus (W_1 \cup U_1)}$.  Hence $T$ is acceptable (\cite[p.~415]{OH}) and the hyperplane of $\RR^d$ defined by $\sum_{\xi \in W_1} x_{\xi} = \sum_{\xi' \in U_1} x_{\xi'}$ is a facet of $q\P_G$.  Now, suppose that $a_i = 0$.  Since $\P_G$ possesses the integer decomposition property, one has $\sum_{\xi \in W_1} a_{\xi} = \sum_{\xi' \in U_1} a_{\xi'}$.  Hence $(a_1, \ldots, a_d) \in \ZZ^d$ cannot belong to $q(\P_G \setminus \partial \P_G)$.  Thus $a_i > 0$, as required.
\end{proof}

We say that a finite subset $L \subset E(G)$ is an {\em edge cover} of $G$ if $\cup_{e \in L}e = [d]$.  Let $\mu(G)$ denote the minimal cardinality of edge covers of $G$.  

\begin{Corollary}
\label{nice}
When $K[G]$ is normal, one has $q \geq \mu(G)$ if $q(\P_G \setminus \partial \P_G) \cap \ZZ^d \neq \emptyset$.
\end{Corollary} 

\begin{proof}
Since $\P_G$ possesses the integer decomposition property, Lemma \ref{facet} guarantees that, if $\ab \in q(\P_G \setminus \partial \P_G) \cap \ZZ^d$, one has $q \geq \mu(G)$.  
\end{proof}

Once Corollary \ref{nice} is established, to complete the proof of Theorem \ref{oldhibi} is a routine job on computing regularity of normal toric rings. 

\begin{proof}[Proof of Theorem \ref{oldhibi}]
In each of the cases (a) and (b), since the edge ring $K[G]$ is normal, it follows that the Hilbert function of $K[G]$ coincides the Ehrhart function (\cite[p.~100]{GTM279}) of the edge polytope $\P_G$, which says that the Hilbert series of $K[G]$ is of the form
\[
(h_0 + h_1 \lambda + \cdots + h_s \lambda^s)/(1 - \lambda)^{(\dim \P_G) + 1}
\]
with each $h_i \in \ZZ$ and $h_s \neq 0$.  One has
\[
s = (\dim \P_G + 1) - \min\{ \, q \geq 1 \, : \, q(\P_G \setminus \partial \P_G) \cap \ZZ^d \neq \emptyset \, \}.
\]
Now, Corollary \ref{nice} guarantees that
\[
s \leq (\dim \P_G + 1) - \mu(G).
\]
Finally, since $\mu(G) = d - \mat(G)$ (\cite[Lemma 2.1]{HeHi}), one has
\[
\reg K[G] = s \leq \dim \P_G - (d - 1) + \mat(G),
\]
as required.
\end{proof}

When $K[G]$ is non-normal, the behavior of regularity is curious.

\begin{Proposition}
\label{newhibi}
For given integers $0\leq r\leq m$, there exists a finite connected simple graph $G$ such that  $\reg K[G]=r$,
and
\[
 \mat(G)=
\begin{cases}
m, \text{if $G$ is non-bipartite},\\
m+1, \text{if $G$  is bipartite}.
\end{cases}
\]
\end{Proposition}

\begin{proof}
In the non-bipartite case, let $H$ be the complete graph with $2r$ vertices. Its matching number is $r$. It is known from \cite[Corollary 2.12]{BVV} that $\reg K[H]=r$. At one vertex of $H$ we attach a path graph of length $2(m-r)$ and call this new graph $G$. Then  $\mat(G)=m$ and $\reg K[G]=\reg K[H]=r$, because $K[G]$ is just a polynomial extension of $K[H]$.

In the bipartite case, let $H$ be the bipartite graph of type $(r+1, r+1)$. Its matching number is $r+1$. Indeed, $K[H]$ may be viewed as Hibi ring whose regularity is well-known, see for example \cite[Theorem 1.1]{EHM}.  At one vertex of $H$ we attach a path graph of length $2(m-r)$ and call this new graph $G$. Then  $\mat(G)=m+1$ and $\reg K[G]=\reg K[H]=r$, by the same reason as before.
\end{proof}

These bounds for the regularity of $K[G]$ are in general only valid, if $K[G]$ is normal. Consider for example the graph $G$ which consists of two disjoint triangles combined be a path of length $\ell$. Then the defining ideal of $K[G]$ is generated by a binomial of degree $\ell +3$, and hence $\reg K[G]=\ell+2$, while the matching number of $G$ is $2+\lceil \ell/2\rceil$.

\begin{Question}
\label{nobodyknows}{\em
Let $m$ be a positive integer, and consider the set  $\mathcal{S}_m$ of finite connected simple graphs with matching number $m$. Is the there are   bound for  $\reg K[G]$ with $G\in \mathcal{S}_m$, and if such bound exists  is it a linear function of $m$?}
\end{Question}

\end{document}